\documentclass[12pt,reqno]{amsart}
\usepackage[utf8]{inputenc}
\usepackage[T1]{fontenc}
\usepackage{graphicx}
\usepackage[export]{adjustbox}
\graphicspath{ {./images/} }
\usepackage{amsmath}
\usepackage{amsfonts}
\usepackage{amssymb}
\usepackage[version=4]{mhchem}
\usepackage{stmaryrd}
\usepackage{amsthm}
\usepackage{xcolor}
\usepackage{enumitem}

\newtheorem{theorem}{Theorem}
\newtheorem{lemma}{Lemma}

\newcommand{\ds}{\displaystyle}

\begin{document}
\author[D. Albanez]{Débora A. F. Albanez}
\address{Departamento Acadêmico de Matemática\\ Universidade Tecnológica Federal do Paraná\\ Cornélio Procópio PR, 86300-000, Brazil} \email{deboraalbanez@utfpr.edu.br}

\author[M. Benvenutti]{Maicon José Benvenutti}
\address{Departamento de Matemática\\ Universidade Federal de Santa Catarina\\ Blumenau SC, 89036-004, Brazil} \email{m.benvenutti@ufsc.br}

\author[S. Little]{Samuel Little}
\address{Department of Applied Physics \\Towson University\\
Towson, MD 21252, USA} \email{slittle2@students.towson.edu}

\author[J. Tian]{Jing Tian}
\address{Department of Mathematics \\Towson University\\
Towson, MD 21252, USA} \email{jtian@towson.edu}

\def\bar{\overline}

\title[Parameter Analysis for 3D Modified Leray-Alpha Model]
{Parameter Error Analysis for the 3D Modified Leray-Alpha Model: Analytical and Numerical Approaches}


\begin{abstract}
In this study, we conduct a parameter error analysis for the 3D modified Leray-$\alpha$ model using both analytical and numerical approaches. We first prove the global well-posedness and continuous dependence of initial data for the assimilated system. Furthermore, given sufficient conditions on the physical parameters and norms of the true solution, we demonstrate that the true solution can be recovered from the approximation solution, with an error determined by the discrepancy between the true and approximating parameters. Numerical simulations are provided to validate the convergence criteria.
\end{abstract}

\maketitle

\section{Introduction}

In recent years, the so-called $\alpha$-models (see \cite{cao2005clark}, \cite{cheskidov2005leray},  \cite{foias2002three}, \cite{ilyin2006modified}, \cite{layton2006well} and references therein)  have attracted significant attention within the fluid dynamics community. These models serve as analytical alternatives to the three-dimensional Navier-Stokes equations (NSE) and provide viable options for computational implementations and simulations. The success of these models as a suitable closure model for turbulence is due to the fact that they conduct as a regularization form of the 3D NSE, involving a lengthscale parameter $\alpha>0$ that is related to a smoothing kernel associated with the Green function of the Helmoltz operator:
$$v=u-\alpha^2\Delta u$$
and we have $v\rightarrow u$ in some sense, as $\alpha\rightarrow 0^+$. In this work, we consider the three-dimensional viscous modified Leray-$\alpha$ (ML-$\alpha$) model
 \begin{equation}
\left\{
\begin{array}{l}
\ds\frac{\partial v}{\partial t}- \nu \Delta v +(v\cdot\nabla)u+\nabla p =  f,   \\
v=(1-\alpha^{2}\Delta)u,\\
\nabla\cdot u  =\nabla\cdot v=  0, \,\, u(0)=u_0,
\end{array}
\right.  \label{1}
\end{equation}
 subject to periodic boundary conditions $\Omega=[0,L]^{3}$, namely, 
\begin{eqnarray}
u(x,t)=u(x+Le_{i},t),\,\, \forall\, (x,t)\in \mathbb{R}^{3}\times \lbrack
0,T ), \label{000}
\end{eqnarray} 
where $e_{1},e_{2}$ and $e_{3}$ is the canonical basis of $\mathbb{R}^{3}$ and $L>0$ is the fixed period. Here, $v$ and $u$  are the nonfiltered and filtered velocities, respectively. Also, $\nu>0$ is the kinematic viscosity, $f=f(x,t)$ is the external force, and $p=p(x,t)$ is the scalar pressure field.

The ML-$\alpha$ model was inspired by the Leray-$\alpha$ model by replacing the nonlinear term $(u\cdot\nabla)v$ in Leray-$\alpha$ model by $(v\cdot \nabla)u$ \cite{ilyin2006modified}. The Leray-$\alpha$ model was first considered by Leray as a general regularization form of the
NSE \cite{leray1934}. It has shown excellent analytical and computational results when using as a closure model in turbulent channels and pipes \cite{cheskidov2005leray, holm1999}. The ML-$\alpha$ model was introduced in \cite{ilyin2006modified} as a closure model of the
Reynolds averaged equations in turbulent channels and pipes. The authors in \cite{ilyin2006modified} showed that when the ML-$\alpha$ model is used as a closure model, it produces the same reduced system of equations as the Navier-Stokes-$\alpha$ model.
The ML-$\alpha$ has many good properties such as a remarkable match to the empirical data for a wide range of
huge Reynolds numbers, the global wellposedness, and a finite dimensional global attractor \cite{ilyin2006modified}. Moreover, for the ML-$\alpha$ model, the steeper slope of the energy spectrum at higher wavenumbers within the inertial range, compared to the traditional $k^{-5/3}$
  Kolmogorov energy spectrum, indicates reduced energy in the higher wavenumber range \cite{ilyin2006modified}. This is consistent with the behavior expected from an effective subgrid-scale model of turbulence.

In this paper, we aim to analyze the behavior of the solutions of ML-$\alpha$ equations when the lengthscale parameter $\alpha>0$ in \eqref{1} is uncertain. To address this, a continuous data assimilation (CDA) technique is used, incorporating spatially discrete observational measurements into the ML-$\alpha$ physical system. In this approach, a ``guess'' parameter $\beta>0$ is used in place of the unknown $\alpha$.
The method of continuous data assimilation was introduced first in \cite{azouani2014continuous} for 2D NSE and later for several other models in different situations, as stochastically noisy data (see \cite{bessaih2015continuous}) and also using observational measurements of only one component of velocity (see \cite {farhat2016abridged}). For more applications of CDA technique, see for instance, \cite{albanez2018continuous}, \cite{albanez2016continuous},
\cite{farhat2015continuous}, \cite{farhat2016data}, \cite{farhat2016charney}, \cite{jolly2017data}, \cite{jolly2015determining}, \cite{jolly2019continuous}, \cite{markowich2016continuous}. 

 To construct an approximate solution for the original ML-$\alpha$ model using the "guess" parameter $\beta$, we employ the following CDA technique applied to the ML-$\alpha$ model:
\begin{equation}
\left\{
\begin{array}{l}
\ds\frac{\partial z}{\partial t}- \nu \Delta z +(z\cdot\nabla)w+\nabla p =  f-\eta(I-\beta^2\Delta)(I_{h}(w)-I_{h}(u)),   \\
z=(1-\beta^{2}\Delta)w,\,\, \nabla\cdot w  =\nabla\cdot z=  0,
\end{array}
\right.  \label{22}
\end{equation}
where $\beta>0$ is the lengthscale parameter. We assume that discrete spatial observational measurements of the real-state solution $u(t)$ of ML-$\alpha$ can be used to construct the linear interpolant operator $I_h$ in \eqref{22}, where $h$ denotes the spatial resolution of the sparse measurements in the data collection process. Moreover, $\eta>0$ represents a nudging parameter, which is fixed and determined based on specific conditions involving the system's physical parameters. Additionally, the interpolant operator $I_{h}:\dot{H}^2{(\Omega)}\longrightarrow \dot{L}^2(\Omega)$ is required to satisfy the approximation property
\begin{equation} 
 \|I_{h}g-g\|^{2}_{L^{2}}\leqslant c_{1}^2h^{2}\|g\|^{2}_{L^{2}} + c_{2}^2h^{4}\|\Delta g\|^{2}_{L^{2}}, \label{LerayProjector}
\end{equation}
where $c_1,c_2>0$ are nondimensional constants.

This work is inspired in \cite{carlson2020parameter}, where the authors applied a similar CDA technique to construct an approximate solution for the 2D Navier-Stokes equations in the absence of the true value of the physical viscosity $\nu>0$. Subsequently, a similar approach was considered for the three-dimensional simplified Bardina and Navier-Stokes-$\alpha$ models in \cite{albanez2024parameter}. In this work, we prove the global well-posedness and continuous dependence of initial data for the assimilated system \eqref{22}, under suitable conditions. Furthermore, given sufficient conditions on the physical parameters and norms of the true solution $u(t)$, we prove that
 $u(t)$ can be recovered through the approximation solution $w(t)$ of \eqref{22}, with an error depending on the difference between the true and approximating parameters $\alpha$ and $\beta$, repectively. To validate our theoretical analysis, we
have presented some numerical simulations. The simulations are powered by
a flexible new Python package, Dedalus.

The paper is organized as follows: In Section \ref{Prelim}, we introduce the functional framework and a priori estimates for Sobolev periodic spaces. Additionally, we present a priori estimates for the ML-$\alpha$ equations, which are used for the analytical analysis. The regularity of solutions for the data assimilation system and the long-time error analysis are present in Section \ref{analyt}. In Section \ref{NumSim}, we provide numerical simulations that validate the long-time error results. Finally, the conclusions are summarized in Section \ref{conclusion}.

\section{ Preliminaries}\label{Prelim}

\label{AAA1}


\subsection{ Basic definitions and inequalities}
\label{sb1}


To formulate the problem, this section introduces some basic definitions, functional settings, and properties.  Let $\Omega=[0,L]^{3}$ represent the periodic box for some period $L>0$. 
We denote $L^{p}(\Omega) $ as the standard three-dimensional Lebesgue vector spaces. Given the assumption of spatial periodicity, the solutions of (\ref{1})-(\ref{000}) satisfy 
\begin{eqnarray}
\frac{d}{dt}\int_{\Omega}u(x,t)\,dx=\int_{\Omega}f(x,t)\,dx. \nonumber
\end{eqnarray}
Considering the forcing $f$ and initial data $u_0$ such that $\ds\int_\Omega f\, dx=\ds\int_\Omega u_0\,dx=0$, we have that the mean of the solution is invariant, which leads us to zero average spaces.\\
For each $s \in \mathbb{R}$, we define the Hilbert space

\begin{eqnarray}
\dot{V}_{s}= \bigg\{&u(x)=\ds\sum_{K \in \mathbb{Z}^3\backslash\{0\}} \hat{u}_{K}e^{2\pi i \frac{K\cdot x}{L}}; \nonumber\\
&\overline{\hat{u}}_{K}=\hat{u}_{-K}, \,\, \ds\frac{K}{L}\cdot \hat{u}_k=0, \,\, \sum_{K \in \mathbb{Z}^3\backslash\{0\}} \bigg|\ds\frac{K}{L}\bigg|^{2s}|\hat{u}_{K}|^{2}<\infty\bigg\},
\end{eqnarray}
endowed with the inner product
\begin{equation}
(u,v)_{\dot{V}_{s}}= L^{3}\sum_{K \in \mathbb{Z}^3\backslash\{0\}} \left(\frac{2\pi|K|}{L}\right)^{2s} \hat{u}_{K}\cdot\overline{\hat{v}_{K}}.
\end{equation}
We have $\dot{V}_{s_{1}} \subset \dot{V}_{s_{2}}$ when $s_{1} \geqslant s_{2}$. For all $s \geqslant 0$, $\dot{V}_{-s}$ is the dual of $\dot{V}_{s}$  (see \cite{Foias_Manley_Rosa_Temam_2001} and \cite{doi:10.1137/1.9781611970050}).

We denote by $\mathcal{P}:\dot{H}_{0}\rightarrow \dot{V}_{0}$ the classical Helmholtz-Leray orthogonal projection given by 
$$ \mathcal{P}u= \sum_{K \in \mathbb{Z}^3\backslash\{0\}} \left(\hat{u}_{K}-\frac{K(\hat{u}_{K} \cdot K)}{|K|^{2}} \right)e^{2\pi i \frac{K\cdot x}{L}}, $$
 and $A: \dot{V}_{2s} \rightarrow \dot{V}_{2s-2}$ the operator given by 
$$Au=\sum_{K \in \mathbb{Z}^3\backslash\{0\}} \frac{4\pi^{2}|K|^{2}}{L^{2}}\hat{u}_{K}e^{2\pi i \frac{K\cdot x}{L}}.$$
We can show that, under periodic boundary conditions, $Au=-\Delta u = -\mathcal{P} \Delta u=-\Delta \mathcal{P} u$ for all $u \in \dot{V}_{2}$.

 By spectral theory, there exists a sequence of eigenfunctions $\{\phi_{K}\}_{K \in \mathbb{N}^3\backslash\{0\}}$ such that
\begin{equation}\label{01}
\left\{\begin{array}{l}
\{\phi_{K}\}\,\, \mbox{ is an  orthonormal basis of }\,\,   \dot{V}_{0},\\
\{\lambda_{K}^{-\frac{1}{2}}\phi_{K}\}\,\, \mbox{ is an orthonormal basis of }\,\, \dot{V}_{1},
\end{array}\right. 
\end{equation}
 where  $\displaystyle \lambda_{K}= \frac{4\pi^{2}|K|^{2}}{L^{2}}$ is the set of eigenvalues of $A: \dot{V}_{0} \rightarrow \dot{V}_{0}$ with domain $D(A)=\dot{V}_{2}$.

We adopt the classical notations  $H= \dot{V}_{0},\ V= \dot{V}_{1}$, and $D(A)'=\dot{V}_{-2}$, besides $|u|=\|u\|_{L^{2}}$, $(u,v)=(u,v)_{L^{2}}$, $\|u\|=\|\nabla u\|_{L^{2}}$, and $((u,v))=(\nabla u,\nabla v)_{_{L^{2}}}$.

 Due to the  Poincar{\'e} inequalities
 \begin{eqnarray}
 |u|^{2}\leqslant\lambda^{-1}_{1}\|u\|^{2}, \,\,\,\forall\,\, u\in V \mbox{  and  }\,\, \|u\|^{2}\leqslant\lambda^{-1}_{1}|Au|^{2},\,\,\, \forall\,\, u\in D(A), \label{hhh111}
\end{eqnarray}
where $\displaystyle \lambda_{1}= \frac{4\pi^{2}}{L^{2}}$,  we have the following equivalent norms
\begin{align}
\|u\|_{H}= |u|, \,\,\,\,\, \|u\|_{V}= \| u\|, \,\,\,\,\, \mbox{and} \,\,\,\,\,\|u\|_{\dot{V}_{2}}= |A u|.
\end{align}
 
 We revisit specific three-dimensional cases of the Gagliardo-Nirenberg inequality  (see \cite{Friedman}):
\begin{equation}
\left\{
\begin{array}{ll}
\Vert g\Vert _{L^{6}}\leqslant c\Vert  g\Vert, & \mbox{ for all  } g\in V; \\
\Vert g\Vert _{L^{3}} \leqslant c|g|^{\frac{1}{2}}\Vert g\Vert^{\frac{1}{2}}, & \mbox{ for all  } g\in V; \\
\Vert g\Vert _{L^{\infty}} \leqslant c\Vert  g\Vert^{\frac{1}{2}}\,  |Ag|^{\frac{1}{2}}, & \mbox{ for all  } g\in \dot{V}_{2},%
\end{array}%
\right.  \label{gn}
\end{equation}%
where  $c$  is a dimensionless constant.

 For each $\alpha>0$, we have
\begin{eqnarray}
(I+\alpha^{2}A)^{-1}u = \sum_{K \in \mathbb{Z}^3\backslash\{0\}} \frac{L^{2}}{L^{2}+4\alpha^{2}\pi^{2} |K|^{2}}\hat{u}_{K}e^{2\pi i \frac{K\cdot x}{L}},
\label{hhh1}
\end{eqnarray}
and also the estimates
\begin{align}
|(I+\alpha^{2}A)^{-1}u| \leqslant |u| \,\,\,\, \mbox{ and } \,\,\,\, |(I+\alpha^{2}A)^{-1}u| \leqslant \frac{1}{4\pi^2\alpha^{2}}\|u\|_{\dot{V}_{-2}}. \label{aa-1}
\end{align}

We define the bilinear form $B:D(A)\times D(A)\rightarrow H$ as the continuous operator 
$$B(v,u)=\mathcal{P}[(v\cdot\nabla) u].$$
 For $u,v, w \in D(A)$, the bilinear term $B$ has the property
\begin{equation}
 ( B(v,u),w )=-( B(v,w),u), \nonumber
 \end{equation}
and hence 
 \begin{equation}\label{BilinearZera}
( B(v,w),w)=0.
 \end{equation}
Moreover,
\begin{equation}\label{bilin4}
|( B(v,u),w)|\leqslant C|v|^{\frac{1}{2}}\|v\|^{\frac{1}{2}}\|u\|\|w\|.
\end{equation}
For every $u, v\in D(A)$ and $w\in H$, we also have
\begin{equation}\label{bilin5}
|(B(v,u),w)|\leqslant C|v| \|u\| \|w\|^{\frac{1}{2}}|Aw|^{\frac{1}{2}},
\end{equation}
which further implies, by duality:
\begin{equation}\label{BI6}
\|B(v,u)\|_{D(A)'} \leqslant C \lambda_{1}^{-\frac{1}{4}}|v|\,\|u\|.
\end{equation}
From \eqref{aa-1} and \eqref{BI6}, we obtain
\begin{equation}\label{bilin6}
|(I+\alpha^{2}A)^{-1}B(v,u)|\leqslant C\frac{\lambda_{1}^{-\frac{1}{4}}}{4\pi^{2}\alpha^{2}}|v|\|u\|.
\end{equation}

We write  (\ref{1}) using functional settings as

\begin{equation}\label{MLalpha}
    \left\{\begin{array}{l}\ds\frac{dv}{dt}+\nu A v+B(v,u)=f, \\ v=u+\alpha^{2} Au, \quad \nabla \cdot v=\nabla \cdot u=0,\end{array}\right.
\end{equation}
with initial condition $u(0)=u_{0}$ and $v(0)=u_{0}+\alpha^{2}Au_{0}$.\\

We revisit the global well-posedness results for three-dimensional viscous modified Leray-$\alpha$ (\ref{1}), as established in \cite{ilyin2006modified}.

\begin{theorem}[Existence and Uniqueness of Regular Solutions \cite{ilyin2006modified}] \label{ExiUniBard}
Let $f\in L^{2}([0,T];H)$, $u(0)=u_{0}\in V$, and $T>0$, the system \eqref{MLalpha} has a unique regular solution $u\in C([0,T);V)\,\cap\, L^{2}([0,T];D(A))$, with $\frac{du}{dt}\in L^{2}([0,T];H)$.
\end{theorem}
In order to establish the result on the long-time error stated in Theorem \ref{teoMLalpha}, we enunciate the following alternative Gronwall's version, whose proof can be found in \cite{albanez2024parameter}.
\begin{lemma}[Gronwall Inequality \cite{albanez2024parameter}]\label{GronwallVersion}
Let $\xi:[t_0,\infty)\rightarrow[0,\infty)$ be an absolutely continuous function and let $\beta:[t_0,\infty)\rightarrow[0,\infty)$ be locally integrable. Assume the existence of positive constants $C,M$ and $T$ such that
\begin{equation}\label{int1}
    \sup\limits_{t\geqslant t_0}\ds\int_t^{t+T}\beta(s)\,ds\leqslant M
\end{equation}
and
\begin{equation}\label{ineq1}
    \ds\frac{d}{dt}\xi(t)+C\xi(t)\leqslant \beta(t)
\end{equation}
  are satisfied for all $t\geqslant t_0$. Then
  \begin{equation}\label{int211}
    \xi(t)\leqslant e^{-C(t-t_0)}\xi(t_0)+M\left(\frac{e^{2CT}}{e^{CT}-1}\right),  
  \end{equation}
  
   for all $t\geqslant t_0$.
\end{lemma}

\subsection{Estimates for the ML-$\alpha$ solutions}
Here, we present some estimates related to the global unique solutions of the ML-$\alpha$ \eqref{MLalpha}.

\begin{lemma}\label{LEM1}
    For all $t\geqslant 0$, we have the following estimate for the global unique solution of \eqref{MLalpha}: 
    
$$|u(t)|^{2}+\alpha^{2}\left\| u(t)\right\|^{2} \leqslant e^{-\nu \lambda_{1} t}\left(|u_{0}|^{2}+\alpha^{2}\| u_{0}\|^{2}\right)+\frac{1}{\lambda_{1}^{2}\nu^2} \sup|f(s)\|_{L}^{2},
$$
where $u_0=u(0)\in V$ is the initial data. Defining
\begin{equation}\label{M1}
M_{1}:=|u_{0}|^{2}+\alpha^{2}\left\| u(0)\right\|^{2}+\frac{1}{\lambda_{1}^{2} \nu^{2}} \sup _{s\geqslant 0 }|f(s)|^{2},    
\end{equation}
it follows that:
\begin{equation}
|u(t)|^{2}+\alpha^{2}\| u(t)\|^{2} \leqslant M_{1}. \quad 
\end{equation}

 Furthermore, for all $T>0$ and $t\geqslant 0$, the following estimate holds:
  $$
\nu \int_{t}^{t+T} \|u(s) \|^{2}+\alpha^{2} | A u (s) |^{2} ds \leqslant | u(t) |^{2}+\alpha^{2} \| u(t)\|^{2}+\frac{T}{\nu \lambda_{1}} \sup _{s \geqslant 0} | f(s) |^{2}.
$$  
In particular,
$$
\int_{t}^{t+T}\|u(s)\|^{2}+\alpha^{2}|A u(s)|^{2} d s \leqslant \frac{M_{1}}{\nu}+\frac{T}{\nu^{2} \lambda_{1}} \sup |f(s)|^{2}.
$$ 
\end{lemma}
\begin{proof}

The proof follows the same steps as the proof of Lemma 2 in \cite{albanez2024parameter} for the Navier-Stokes-$\alpha$ equations, since $(B(v, u), u)=(\tilde{B}(u, v), u)=0$. 
\end{proof}

 \begin{lemma}\label{LEM2}
For all $T>0$ and $t\geqslant 0$, we have the following estimate for the time-derivative of the solutions of \eqref{MLalpha}:
\begin{align}\label{DER}
   \ds\int_t^{t+T}\left|\frac{d u}{d t}(s)\right|^2ds\leqslant
   &\left(\ds\frac{6 M_{1}}{\nu \alpha^{2}}+\frac{6T}{\nu^{2} \lambda_{1} \alpha^{2}} \sup\limits_{s \geqslant 0} |f(s)|^{2} \right) \left(\nu^2 +\ds\frac{c^2 M_{1}}{16 \pi^{4} \lambda_{1}^{1 / 2} \alpha^{2}} \right)\nonumber\\
    &+\ds\frac{3 c^{2}M_{1}^2T}{16 \pi^{4} \alpha^{6} \lambda_{1}^{1 / 2}} + 3T \sup\limits_{s \geqslant 0}|f(s)|^2,  
\end{align}
where $M_1$ is defined in \eqref{M1}.
  \end{lemma}
\begin{proof}
     
  Applying $\left(I+\alpha^{2} A\right)^{-1}$ on \eqref{MLalpha}, we have the following ML-$\alpha$ equivalent system:
$$
\frac{du}{d t}+\nu A u+\left(I+\alpha^{2} A\right)^{-1} B(v, u)=\left(I+\alpha^{2} A\right)^{-1} f.
$$

Using \eqref{bilin6}, we have

$$
\begin{aligned}
\left|\frac{d u}{d t}\right| 
& \leqslant \nu|A u|+|\left(I+\alpha^{2} A\right)^{-1} B(v, u)|+|\left(I+\alpha^{2} A\right)^{-1} f| \\
& \leqslant \nu|A u|+\frac{c}{4 \pi^{2} \alpha^{2} \lambda_{1}^{1 / 4}}|v|\,\| u\|+|f| \\
& \leqslant \nu|A u|+\frac{c}{4 \pi^{2} \alpha^{2} \lambda_{1}^{1 / 4}}\|u\|\left(|u|+\alpha^{2}|A u|\right)+|f|.
\end{aligned}
$$
Therefore
\begin{equation*}
 \left|\frac{d u}{d t}\right| \leqslant\left|A u\right|\left(\nu+\frac{c}{4 \pi^{2} \lambda_{1}^{1 / 4}}\|u\|\right)+\frac{c}{4 \pi^{2} \alpha^{2}\lambda_{1}^{1 / 4}} \| u\|\, | u|+| f |, 
\end{equation*}
which yields
\begin{equation*}
 \left|\frac{d u}{d t}\right|^2\leqslant  6|A u|^{2}\left(\nu^{2}+\frac{c^{2}}{16 \pi^{4} \lambda_{1}^{1 / 2}}\|u\|^{2}\right) +\frac{3 c^{2}}{16 \pi^{4} \alpha^{4} \lambda_{1}^{1 / 2}}\|u\|^{2}|u|^{2}+3|f|^{2}.
\end{equation*}
Finally, integrating from $t$ to $t+T$ and using Lemma \ref{LEM1}, we have 
\begin{eqnarray}
   \ds\int_t^{t+T}\left|\frac{d u}{d t}(s)\right|^2ds\leqslant 
   & \left(\ds\frac{6 M_{1}}{\nu \alpha^{2}}+\frac{6T}{\nu^{2} \lambda_{1} \alpha^{2}} \sup\limits_{s \geqslant 0} |f(s)|^{2} \right) \left(\nu^2 +\ds\frac{c^2 M_{1}}{16 \pi^{4} \lambda_{1}^{1 / 2} \alpha^{2}} \right)\nonumber\\
    &+ \left(\ds\frac{3 c^{2}}{16 \pi^{4} \alpha^{4} \lambda_{1}^{1 / 2}} \cdot \ds\frac{M_{1}}{\alpha^{2}} \cdot M_{1}\right) T + 3T \sup\limits_{s \geqslant 0}|f(s)|^2.  \nonumber
\end{eqnarray}
This leads to \eqref{DER}.

\end{proof}

\section{Regularity Analysis and Error Estimates}\label{analyt}

\subsection{Regularity analysis for the data assimilation system}
Using the functional setting again, the continuous data assimilation system (\ref{22}) is equivalent to
\begin{equation}\label{MLalphaCDA}
 \left\{
 \begin{array}{l}
\ds\frac{dz}{dt}+\nu Az+B(z, w)=f-\eta(I+\beta^{2} A)\mathcal{P}(I_{h} w-I_{h} u), \\
 z=w+\beta^{2} A w, \nabla \cdot z=\nabla \cdot w=0,
 \end{array}\right.   
\end{equation}
with initial condition $w(0)=w_{0}$ and $z(0)=w_{0}+\beta^{2}Aw_{0}$.

We now state the global well-posedness result for the data assimilation system (\ref{MLalphaCDA}).

\begin{theorem}[Global well-posedness]\label{aaa1} 
Let $T>0$, and consider $f\in L^{\infty}([0,T];H)$, with $u$ as the regular solution of (\ref{MLalpha}) and initial data $u_{0}\in V$. Let $w_{0}\in V$ and $\eta>0$ given. Moreover, suppose $I_{h}:\dot{H}^2{(\Omega)}\longrightarrow \dot{L}^2(\Omega)$ is the linear interpolant operator satisfying (\ref{LerayProjector}) and the following conditions are valid:
\begin{eqnarray}
 \eta\leq \frac{\nu}{2} \min \bigg\{\ds\frac{1}{c_1^2h^2},\frac{\beta^2}{c_2^2h^4}\bigg\}   \mbox{ and } \eta^2c_2^2h^4<\ds\frac{\nu^2}{4}, \label{kkkk}
\end{eqnarray}
where $c_1,c_2>0$ are given in \eqref{LerayProjector}. Under these assumptions, the continuous data assimilation system (\ref{MLalphaCDA}) has a unique solution with the regularity
\begin{equation}
 w\in C([0,T);V)\cap L^{2}([0,T];D(A))\,\,\mbox{  and  }\,\,\frac{dw}{dt}\in L^{2}([0,T];H).
\end{equation}
Finally, there exists a continuous dependence with respect to initial data $w_{0}$ in $V$-norm. 
\end{theorem}

\begin{proof}
First, we apply the bounded operator $(I+\beta^2A)^{-1}$ in \eqref{MLalphaCDA} and obtain
\begin{equation}\label{EqOp}
 \ds\frac{dw}{dt}+\nu Aw+(I+\beta^2A)^{-1}B(z, w)=(I+\beta^2A)^{-1}f-\eta\mathcal{P}(I_{h} w-I_{h} u).   
\end{equation}

Note that proving the existence of solutions to \eqref{EqOp} is equivalent to proving the existence of solutions to \eqref{MLalphaCDA}. They can be obtained via the standard Galerkin procedure, using a basis of eigenfunctions with properties (\ref{01}). We denote
$$F(t)=(I+\beta^2A)^{-1}f(t)+\eta\mathcal{P}I_{h}u(t).$$
Since $u\in L^{2}([0,T];D(A))$ and property (\ref{LerayProjector}) is valid, we have $F\in L^{2}([0,T];H)$. Consider the linear spanned space $H_{m}=[\phi_{1},...,\phi_{m}]$, the projection operators $P_{m}:H \longrightarrow H_{m}$, and the approximated problem 
\begin{equation}\label{22p}
\left\{
\begin{array}{l}
\ds\frac{dw_{m}}{dt}+P_{m}(I+\beta^2A)^{-1}B(z_{m},w_{m})+\nu A w_{m} 
= P_{m}F-\eta P_{m}\mathcal{P}I_{h}(w_{m}), \\
w_{m}(0)=P_{m}(w_{0}), 
\end{array}
\right. 
\end{equation}%
 where $z_{m}=w_{m}+\beta^{2}Aw_{m}$ and $\displaystyle w_{m}(x,t)=\sum_{j=1}^{m}g_{j,m}(t)\phi_{j}(x)$. By applying the classical ODE theory, we can obtain the existence and uniqueness over a short time $[0,T_{m})$. Subsequently, we prove now uniform bounds for $w_{m}$ independently of $m$, which guarantees existence in time  $[0,T)$ of each $w_{m}$. Taking the dual action $D(A)'$ on $w_m$ in \eqref{22p}, we obtain

 \begin{align}\label{1D}
    &\frac{1}{2} \frac{d}{dt} |w_m|^2 + \nu \|w_m\|^2 
    + \langle P_{m}(I + \beta^2 A)^{-1} B(z_{m}, w_{m}), w_m \rangle_{D(A)'} \\ \nonumber
    =&(P_m F, w_m) - \eta (\mathcal{P} I_h w_m, w_m).
\end{align}
 
Moreover, taking the $L^2$-inner product with $\beta^2 Aw_m$ in \eqref{22p}, we have
 \begin{align}\label{2DD}
     &\ds\frac{1}{2}\frac{d}{dt}\beta^2\|w_m\|^2+\nu\beta^2|Aw_m|^2+( P_{m}(I+\beta^2A)^{-1}B(z_{m},w_{m}),\beta^2Aw_m)\\\nonumber
     =&(P_mF,\beta^2Aw_m)-\eta(\mathcal{P}I_hw_m,\beta^2Aw_m).
 \end{align}
 Adding \eqref{1D} and \eqref{2DD}, we obtain
 \begin{align*}
&\frac{1}{2}\frac{d}{dt}\left(|w_{m}|^{2}+\beta^{2}\|w_{m}\|^{2}\right)+\nu \left(\|w_{m}\|^{2}+\beta^{2}|Aw_{m}|^{2}\right)\\
=&(P_mF,w_{m}+\beta^2Aw_m)-\eta (I_{h}(w_{m}),w_{m}+\beta^2Aw_m),
\end{align*}
where we use the self-adjointness of $(I+\beta^2A)$ operator, the symmetry of $\mathcal{P}$, and property \eqref{BilinearZera}. Applying Hölder, \eqref{LerayProjector}, and Young's inequalities, it yields
$$\begin{aligned}
& \frac{1}{2} \frac{d}{d t}\left(\left|w_m\right|^2+\beta^2\left\|w_m\right\|^2\right)+\nu\left(\left\|w_m\right\|^2+\beta^2\left|A w_m\right|^2\right) \\
\leqslant&\left(\frac{1}{\eta}+\frac{\beta^2}{\nu}\right)|F|^2+\frac{\eta}{4}\left|w_m\right|^2+\frac{\nu}{4} \beta^2\left|A w_m\right|^2 \\
&+\frac{\eta}{4}\left|w_m\right|^2+\eta c_1^2 h^2\left\|w_m\right\|^2-\eta\left|w_m\right|^2 \\
&+\frac{\eta}{2} \beta^2\left\|w_m\right\|^2+\frac{\eta c_1^2 h^2}{2} \beta^2\left|A w_m\right|^2-\eta \beta^2\left\|w_m\right\|^2.
\end{aligned}$$
Moreover, using the hypothesis that $h$ is sufficiently small such that $\eta c_1^2h^2<\ds\frac{\nu}{2}$, we have
\begin{align}\label{2C}
    &\frac{1}{2} \frac{d}{d t}(\left|w_m\right|^2+\beta^2\left\|w_m\right\|^2)+\ds\frac{\nu}{2}\left(\left\|w_m\right\|^2+\beta^2\left|A w_m\right|^2\right)\nonumber\\
     \leqslant &\left(\ds\frac{1}{\eta}+\frac{\beta^2}{\nu}\right)|F|^2-\ds\frac{\eta}{2}(|w_m|^2+\beta^2\|w_m\|^2).
 \end{align}
Using Poincaré inequality and Grönwall standard inequality (see \cite{evans2022partial}) on \eqref{2C}, we have 
\begin{align}\label{2D}
&\left|w_m\right|^2+\beta^2\left\|w_m\right\|^2\\ \nonumber
 \leqslant&(|w_{m}(0)|^{2}+\beta^{2}\|w_{m}(0)\|^{2})+ 2\left(\frac{1}{\eta}+\frac{\beta^2}{\nu}\right)\ds\int_0^T|F(s)|^2ds\\ \nonumber
 =:&E_1(T),
\end{align}
for all $t\in[0,T_m)$. Therefore we have the global existence of $w_m$ in time, since the right-hand side of \eqref{2D} is bounded and the estimate is uniform in $m$ and $t$. Additionally, by integrating \eqref{2C}, we attain the following estimate:
\begin{multline} \label{2E}  
    \ds\int_{0}^{t}(\|w_{m}(s)\|^{2}+\beta^{2}|Aw_{m}(s)|^{2})ds\\
\leqslant \ds\frac{1}{\nu} (|w_{m}(0)|^{2}+\beta^{2}\|w_{m}(0)\|^{2})+\frac{2}{\nu}\left(\ds\frac{1}{\eta}+\frac{\beta^2}{\nu}\right)\ds\int_0^T|F(s)|^2ds=:E_2(T). 
\end{multline}

Thus, from estimates \eqref{2D} and \eqref{2E}, we conclude that
\begin{align}\label{lbounded}
 &\|w_m\|^2_{L^{\infty}([0,T];V)}\leqslant\ds \frac{E_1(T)}{\beta^2},\,  
 \|z_m\|^2_{L^{\infty}([0,T];V')}\leqslant E_1(T),\\ \nonumber
&\|w_m\|^2_{L^2([0,T];V)}\leqslant E_2(T), \, 
\|w_m\|^2_{L^2([0,T];D(A))}\leqslant\ds \frac{E_2(T)}{\beta^2},\\ \nonumber
& \|z_m\|^2_{L^{2}([0,T];H)}\leqslant E_2(T).    
\end{align}
To apply the Aubin-Lion Theorem \cite{Lions69}, we first establish uniform estimates in $m$ for $\frac{dw_{m}}{dt}$. We now revisit the equivalent equation
\begin{align*}
 &\frac{d}{d t}(w_m(t)+\beta^2Aw_m(t))+v A(w_m(t)+\beta^2Aw_m(t))+P_mB\left(z_m, w_m\right)\\
 =&P_m f+\mathcal{P} I_h(u(t))+\beta^2P_m A \mathcal{P} I_h(u(t))\\
&-\eta P_m \mathcal{P} I_h\left(w_m(t)\right) -\eta \beta^2 P_m A \mathcal{P} I_h\left(w_m(t)\right).
\end{align*}
Since $D(A)= \dot{V}_{2}$ and its dual is $\dot{V}_{-2}$, using Gagliardo-Nirenberg inequality (\ref{gn}), from equality above we have
\begin{eqnarray}
\left<\frac{d}{dt}(w_{m}+ \beta^{2}Aw_m),g\right>_{\dot{V}_{-2}, \dot{V}_{2}} &\leqslant& 
C\psi_{m} |Ag|,  \nonumber
\end{eqnarray}%
 for all $g\in D(A)$, where 
 \begin{eqnarray}
\psi_{m}:= &\nu\beta^{2}|Aw_{m}|+\nu|w_{m}|+\frac{1}{\lambda_{1}^{\frac{1}{4}}}\|w_{m}\|(|w_{m}|+|Aw_{m}|)+\ds\frac{P_mf}{\lambda_1}\nonumber\\
&+ \frac{|P_m\mathcal{P}(I_hu)|}{\lambda_1}+\eta\beta^2|\mathcal{P}(I_hu)|+|\mathcal{P}I_{h}w_{m}|+\eta\beta^2|P_m\mathcal{P}I_hw_m|.     
 \end{eqnarray}

Note that $\|\psi_{m}\|_{L^{2}([0,T];H)}$ is bounded uniformly in $m$ by \eqref{lbounded}. Therefore, we conclude that $\frac{d}{dt}z_{m}= \frac{d}{dt}(w_{m}+ \beta^{2}Aw_m)$ is bounded uniformly in ${L^{2}([0,T];D(A))}$. Using (\ref{aa-1}), we also have that $\frac{d}{dt}w_{m}$ is bounded uniformly in ${L^{2}([0,T];H)}$. By applying the Aubin-Lions compactness theorem and Banach-Alaoglu Theorem, we obtain a subsequence of the approximated solutions, denoted by $\{w_{m}\}_{m \in \mathbb{N}}$, such that
\begin{eqnarray}
&&w_{m}\rightarrow w \,\mbox{ weakly in }\, L^{2}([0,T],D(A)),\nonumber\\
&&w_{m}\rightarrow w \,\mbox{ strongly in }\, L^{2}([0,T];V),\nonumber\\
&&\frac{dw_{m}}{dt}\rightarrow \frac{dw
}{dt}\,\mbox{ weakly in }\, L^{2}([0,T],H).\nonumber
\end{eqnarray}
Moreover, for the non-filtered velocity, we have
\begin{eqnarray}
&&z_{m}\rightarrow z \,\mbox{ weakly in }\, L^{2}([0,T],H),\nonumber\\
&&z_{m}\rightarrow z \,\mbox{ strongly in }\, L^{2}([0,T];\dot{V}_{-1}),\nonumber\\
&&\frac{dz_{m}}{dt}\rightarrow \frac{dz}{dt}\,\mbox{ weakly in }\, L^{2}([0,T],\dot{V}_{-2}).\nonumber
\end{eqnarray}
Now, it is straightforward to pass the weak limit in (\ref{22p}) and conclude that $w$ is a solution of \eqref{MLalphaCDA}.

Next, we prove the continuous dependence of solutions on the initial data, which implies the uniqueness of solutions.
Let $w$ and $\tilde{w}$ be two solutions of \eqref{MLalphaCDA}, and denote $W= w- \tilde{w}$ and $Z= z- \tilde{z}$. Subtracting the equations, we get
\begin{align*}
&\frac{d}{dt}(W+\beta^{2}AW)+\nu A(W+\beta^{2}AW)+B(z,w)-B(\tilde{z},\tilde{w})\\
=&-\eta\mathcal{P}I_{h}W-\eta\beta^2A\mathcal{P}(I_hW). 
\end{align*}
Since $B(z,w)-B(\tilde{z},\tilde{w})=B(Z,W)+B(\tilde{z},W)+B(Z,\tilde{w})$, we have
\begin{align}\label{PIs}
 &\frac{1}{2}\frac{d}{dt}(|W|^{2}+\beta^{2}\|W\|^{2})+\nu(\|W\|^{2}+\beta^{2}|AW|^{2})+   (B(Z,\tilde{w}),W)\\ \nonumber
=&-\eta(\mathcal{P}I_{h}W,W)-\eta\beta^2(\mathcal{P}I_hW,AW).
\end{align}

Using the Gagliardo-Nirenberg inequality (\ref{gn}) and Young inequality, we have
\begin{align}\label{BZ}
&(B(Z,\tilde{w}),W)\\ \nonumber
=&(B(W,\tilde{w}),W)+\beta^2(B(AW,\tilde{w}),W)\\ \nonumber
\leqslant& c|W|\,\|\tilde{w}\|\,\|W\|^{\frac{1}{2}}|AW|^{\frac{1}{2}}+c\beta^2|AW|\,\|\tilde{w}\|_{L^3}\|w\|_{L^6}\\ \nonumber
\leqslant&\ds\frac{\nu\beta^2}{8}|AW|^2+\ds\frac{2c^2}{\lambda_1^{1/2}\nu\beta^2}\|\tilde{w}\|^2|W|^2+\ds\frac{\nu\beta^2}{8}|AW|^2+\ds\frac{2c^2\beta^2}{\nu\lambda_1^{1/2}}\|\tilde{w}\|^2\|W\|^2. 
\end{align}

From (\ref{LerayProjector}) and given hypothesis $\eta c_{1}^2h^{2} < \ds\frac{\nu}{2}$ and $\eta c_{2}^2h^{4} < \ds\frac{\nu\beta^{2}}{2}$, we have 
\begin{align}\label{IHW}
&-\eta(\mathcal{P}I_{h}W,W)\\
=& -\eta(\mathcal{P}I_{h}W-W,W) - \eta|W|^{2}\leqslant \frac{\eta}{2}|\mathcal{P}I_{h}W-W|^{2}  - \frac{\eta}{2} |W|^{2}  \nonumber\\
\leqslant& \frac{\eta c_{1}^2h^{2}}{2}\|W\|^{2} + \frac{\eta c_{2}^2h^{4}}{2}|A W|^{2} - \frac{\eta}{2} |W|^{2}\nonumber\\
\leqslant & \frac{\nu}{4}(\|W\|^2+\beta^2|AW|^2)-\frac{\eta}{2} |W|^{2}.
\end{align}

Moreover, from the third condition $\eta^2c_2^2h^4<\ds\frac{\nu^2}{4}$, we also obtain

\begin{eqnarray}\label{third}
    -\eta \beta^2\left(P I_h W, A W\right)&=&-\eta \beta^2\left(P I_h W-W, A W\right)-\eta\beta^2\|W\|^2 \nonumber\\
& \leqslant& \eta \beta^2\left|P I_h W-W\right||A W|-\eta \beta^2\|W\|^2\nonumber \\  &\leqslant& \frac{\nu}{4} \beta^2|A W|^2+\frac{\eta^2 \beta^2}{\nu}\left|P I_h W-W\right|^2-\eta \beta^2\|W\|^2\nonumber \\
&\leqslant& \ds\frac{\nu}{2}\beta^2|AW|^2-\ds\frac{\eta\beta^2}{2}\|W\|^2.
\end{eqnarray} 
 Therefore, from \eqref{PIs}, \eqref{BZ}, \eqref{IHW} and \eqref{third}, we obtain
\begin{eqnarray*}
\frac{d}{dt}\left(|W|^{2} +\beta^{2}\|W\|^{2}\right)&\leqslant&\ds\frac{4c^2}{\lambda_1^{1/2}\nu\beta^2}\|\tilde{w}\|^2|W|^2+\ds\frac{4c^2\beta^2}{\nu\lambda_1^{1/2}}\|\tilde{w}\|^2\|W\|^2\\
&\leqslant & 2\bigg( \ds\frac{c^2}{\lambda_1^{1/2}\nu\beta^2}+\ds\frac{c^2}{\nu\lambda_1^{1/2}}\bigg)\|\tilde{w}\|^2(|W|^{2} +\beta^{2}\|W\|^{2}).
\end{eqnarray*}
Using the classical Gronwall inequality, we conclude that
\begin{equation}
|W(t)|^{2} +\beta^{2}\|W(t)\|^{2}\leqslant\left(|W_{0}|^{2} +\beta^{2}\|W_{0}\|^{2}\right)  \exp\bigg(4\tilde{C}\ds\int_{0}^{t}\|\tilde{w}(s)\|^{2}ds\bigg),  \nonumber
\end{equation}
where $\tilde{C}=\ds\frac{c^2}{\lambda_1^{1/2}\nu\beta^2}+\ds\frac{c^2}{\nu\lambda_1^{1/2}}$.

Due to the regularity of $\tilde{w}$, we can establish the continuous dependence of the regular solution.
\end{proof}

\subsection{Long-time error analysis}
  In the following theorem, we prove that, under suitable conditions on the parameters of the systems \eqref{MLalpha} and \eqref{MLalphaCDA},  the approximate solution $w(t)$ can be used to recover the original ML-$\alpha$ solution, with an error controlled by the difference of $\alpha$ and $\beta$.   

\begin{theorem}\label{teoMLalpha}
Let $u_0,w_0\in V$ and $u$ and $w$ be solutions to systems \eqref{MLalpha} and \eqref{MLalphaCDA}, respectively, with initial data $u_0$ and $w_0$ in $V$. Assume that the following conditions are met:
\begin{enumerate}
    \item $ \ds\frac{3 \eta}{4}-\frac{15^{3} M_{1}^{2}c^4}{2^{11} \nu^{3} \alpha^{4}}>0, $
    \item $ \eta c_{1}^2 h^{2}+\ds\frac{5 \eta^{2} \beta^{2}}{2 \nu} c_{1}^2 h^{2}+\frac{30^{3} c^{4} M_{1}^{2}}{4^{4} \nu^{3} \alpha^{4}}-\eta \beta^{2}<\frac{\nu}{4}, $
    \item $ \eta c_{2}^2 h^{4}+\ds\frac{5 c_{2}^2 \eta^{2} \beta^{2} h^{4}}{2 \nu}<\frac{\nu \beta^{2}}{4}, $
\end{enumerate}
where $c$ is given in \eqref{gn}, $M_{1}$ in \eqref{M1}, and $c_1$ and $c_2>0$ in \eqref{LerayProjector}. Then, the following inequality for the difference between the physical and assimilated solutions, denoted by $g(t):=w(t)-u(t)$, holds for all  $t \geqslant 0$:
$$
|g(t)|^{2}+\beta^{2}\|g(t)\|^{2} \leqslant e^{-\frac{\lambda \nu t}{2}}\left(|g(0)|^{2}+\beta^{2}\|g(0)\|^{2}\right)+M_{\alpha}\left(\frac{e}{e^{1 / 2}-1}\right),
$$
where $M_{\alpha}$ is defined in \eqref{M2}.
\end{theorem}

\begin{proof}
    
Let $g(t)=w(t)-u(t)$. By expressing
\begin{equation*}
B(z, w)-B(v, u)=B(z, g)+B(z-v, u), 
\end{equation*}
and

$$z-v=g+\beta^{2} A g+\left(\beta^{2}-\alpha^{2}\right) Au,$$
it follows that
\begin{equation}\label{B}
B(z, w)-B(v, u)=B(z, g)+B(g, u)+\beta^{2} B(A g, u)+\left(\beta^{2}-\alpha^{2}\right) B(Au, u).
\end{equation}
Subtracting \eqref{MLalpha} from \eqref{MLalphaCDA}, we have
\begin{align*}
  &\ds\frac{d}{dt} [g+\beta^{2} A g +(\beta^{2}-\alpha^{2}) Au]+\nu A\left[g+\beta^{2} Ag+(\beta^{2}-\alpha^{2}\right) Au] \\
     & +B(z, w)-B(v, u)\\
     =&-\eta\left(I+\beta^{2} A\right) I_h g.
\end{align*}

Taking the dual action $\langle \cdot, g\rangle_{D(A)'}$ and using \eqref{BilinearZera} and \eqref{B}, we get

\begin{align*}
 &\frac{1}{2} \frac{d}{dt}\big(|g|^{2}+\beta^{2}\left\|g\right\|^{2}\big)+\nu\left(\left\|g\right\|^{2}+\beta^{2}|Ag|^{2}\right)\\
 =&\left(\alpha^{2}-\beta^{2}\right)\left\langle\frac{d}{d t} A u, g\right\rangle
 -\nu\left(\beta^{2}-\alpha^{2}\right)(Au, A g)-(B(g, u), g)\\
 &-\beta^{2}(B(A g, u), g)
 -\left(\beta^{2}-\alpha^{2}\right)(B(A u, u), g)-\eta\left(I_{h} g, g\right)-\eta \beta^{2}\left(I_h g, Ag\right),
\end{align*}

where we use $\left\langle A P I_{h} g, g\right\rangle_{D(A)'}=\left(I_{h} g, A g\right)$. Thus, 

\begin{align}\label{PI}
 &\ds\frac{1}{2} \frac{d}{d t}\big(|g|^{2}+\beta^{2}\|g\|^{2}\big)+\nu\left(\left\|g \right\|^{2}+\beta^{2}|Ag|^{2}\right)\nonumber \\
\leqslant &\left | \beta^2-\alpha^2 \right | \left | \left \langle \ds\frac{d}{dt} Au, g \right \rangle  \right | +\nu \left | \beta^2-\alpha^2 \right | | Au |\, | Ag | +\left | (B(g,u),g) \right |  \nonumber\\
& +\beta^{2}|(B(A g, u), g)|+\left|\beta^{2}-\alpha^{2}\right|\, |B(A u, u), g\big)| \\
& -\eta\left(I_{h} g-g, g\right)-\eta|g|^{2}-\eta \beta^{2}\left(I_{h} g-g, A g\right)-\eta \beta^{2}\left(g, Ag\right). \nonumber
\end{align}
Now, we estimate the right-hand side terms in \eqref{PI}:

\begin{enumerate}[label=\roman*)]
    \item 
$$\begin{aligned}
\left|\beta^{2}-\alpha^{2}\right|\left | \left \langle \frac{d}{dt} Au,g \right \rangle \right | 
& \leqslant \left|\beta^{2}-\alpha^{2}\right|\bigg|\ds\frac{du}{dt}\bigg|\,|Ag| \\
& \leqslant \frac{\nu \beta^{2}}{10}|Ag|^{2}+\frac{5}{2 \nu \beta^{2}}\left|\beta^{2}-\alpha^{2}\right|^{2}| u_{t}|^{2},
\end{aligned}$$

\item 
$$\begin{aligned}
\nu|\beta^{2}-\alpha^{2}|Au|\,|Ag| & =\sqrt{\nu}|Ag| \sqrt{\nu}|\beta^{2}-\alpha^{2}|Au| \\
& \leqslant \frac{\nu \beta^{2}}{10} |Ag|^{2}+\frac{5\nu}{2 \beta^{2}} \left|\beta^{2}-\alpha^{2}\right|^{2}|A u|^{2},
\end{aligned}
$$
\item $$
\begin{aligned}
\left|(B(g,u), g)\right|
&\leqslant \|g \|_{L^4}\| u\|\,|g|_{L^4}=\|g\|_{L^4}^{2}\| u\| \\
&\leqslant c|g|^{1 / 2}\left\| g\right\|^{3 / 2}\left\| u\right\| \\
&\leqslant \frac{2 \nu}{5}\|g\|^{2}+c^{4} \cdot \frac{15^{3}}{2^{11} \nu^{3}}\| u\|^{4}|g|^{2},
\end{aligned}
$$
\item $$
\begin{aligned}
\beta^{2}|(B(A g, u), g)| 
& \leqslant c \beta^{2}|A g|\,\| u\|\,\|g\|_{L^\infty} \\
& \leqslant c \beta^{2}|A g|\, \| u\|\|g\|^{1/2}|Ag|^{1/2}  \\
& =c \beta^{2}|A g|^{3 / 2}\left\| u \right\|\|g\|^{1 / 2} \\
& \leqslant \frac{\nu \beta^{2}}{10}|Ag|^{2}+c^{4} \cdot \frac{30^{3}}{4^{4} \cdot \nu^{3}}\| u\|^{4}\left\|g\right\|^{2},
\end{aligned}
$$ 
    \item $$
\begin{aligned}
\left|\beta^{2}-\alpha^{2}\right| \left|(B(A u, u), g)\right|
& \leqslant c\left|\beta^{2}-\alpha^{2}\right||Au |\| u\|\left\|g\right\|^{1 / 2}|Ag|^{1 / 2} \\
& \leqslant \frac{\nu \beta}{5}\left\|g\right\||Ag|+\frac{5}{4 \nu \beta} c^{2}\left|\beta^{2}-\alpha^{2}\right|^{2}|Au|^{2}\| u\|^{2} \\
& \leqslant \frac{\nu}{10}\left\|g\right\|^{2}+\frac{\nu \beta^{2}}{10}|A g|^{2}+\frac{5}{4 \nu \beta} c^{2}\left|\beta^{2}-\alpha^{2}\right|^{2}|Au|^{2}\| u\|^{2},
\end{aligned}
$$
\item $$
\begin{aligned}
-\eta\left(I_{h} g-g, g\right)-\eta|g|^{2} 
& \leqslant \eta|I_{h} g-g|^{2}+\frac{\eta}{4}|g|^{2}-\eta|g|^{2} \\
& \leqslant \eta\left(c_{1}^2 h^{2}\left\|g\right\|^{2}+c_{2}^2 h^{4}|A g|^{2}\right)-\frac{3}{4} \eta|g|^{2},
\end{aligned}
$$
\item $$
\begin{aligned}
&-\eta \beta^{2}\left(I_{h} g-g, Ag\right)-\eta \beta^{2}\left(g, Ag\right) \\
& \leqslant \frac{\nu \beta^{2}}{10}|Ag|^{2}+\frac{5\eta^{2} \beta^{2} }{2 \nu}\left(c_{1}^2 h^{2}\left\|g\right\|^{2}+c_{2}^2 h^{4}|Ag|^{2}\right)-\eta \beta^{2}\left\|g\right\|^{2}.
\end{aligned}
$$
\end{enumerate}

Therefore, considering the estimates i) - vii) in \eqref{PI}, we get

$$
\begin{aligned}
&\frac{1}{2} \frac{d}{d t}\left(|g|^{2}+\beta^{2}\left\|g\right\|^{2}\right)+\left(\frac{\nu}{2}-\frac{30^{3} c^{4}}{4^{4} \nu^{3}}\|u\|^{4}-\eta c_{1}^2 h^{2}-\frac{5 \eta^{2} \beta^{2}}{2 \nu} c_{1}^2 h^{2}+\eta \beta^{2}\right)\|g\|^{2} \\
& +\left(\frac{\nu \beta^{2}}{2}-\eta c_{2}^2 h^{4}-\frac{5 \eta^{2} \beta^{2} c_{2}^2 h^{4}}{2 \nu}\right)|Ag|^{2}  +\left(\frac{3 \eta}{4}-\frac{15^{3}c^4}{2^{11} \nu^{3}}\|u\|^{4}\right)|g|^{2} \\
\leqslant &\frac{5}{2 \nu \beta^{2}}\left|\beta^{2}-\alpha^{2}\right|^{2}| u_t|^2 +\frac{5\nu}{2 \beta^{2}}\left|\beta^{2}-\alpha^{2}\right|^{2}| Au |^2 +\frac{5 c^2}{4 \nu \beta^{2}}\left|\beta^{2}-\alpha^{2}\right|^{2}| Au|^2 \| u  \| ^2.
\end{aligned}
$$

We require in hypothesis $\eta>>1$ and $h<<1$ such that
\begin{enumerate}
    \item $ \ds\frac{3 \eta}{4}-\frac{15^{3} M_{1}^{2}c^4}{2^{11} \nu^{3} \alpha^{4}}>0, $
    \item $ \eta c_{1}^2 h^{2}+\ds\frac{5 \eta^{2} \beta^{2}}{2 \nu} c_{1}^2 h^{2}+\frac{30^{3} c^{4} M_{1}^{2}}{4^{4} \nu^{3} \alpha^{4}}-\eta \beta^{2}<\frac{\nu}{4}, $
    \item $ \eta c_{2}^2 h^{4}+\ds\frac{5 c_{2}^2 \eta^{2} \beta^{2} h^{4}}{2 \nu}<\frac{\nu \beta^{2}}{4}, $
\end{enumerate}

and from Poincare's inequality and conditions above, we obtain

\begin{align}\label{Grn}
&\frac{1}{2} \frac{d}{d t}\big(|g|^{2}+\beta^{2}\|g\|^{2}\big)+\ds\frac{\nu \lambda_{1}}{4}\left(|g|^{2}+\beta^{2}\|g\|^{2}\right) \nonumber\\
\leqslant &\ds\frac{5}{2 \nu \beta^{2}}\left|\beta^{2}-\alpha^{2}\right|^{2}|u_{t}|^{2}+\frac{5 \nu}{2 \beta^{2}}\left|\beta^{2}-\alpha^{2}\right|^{2}|Au|^{2}+\ds\frac{5 c^{2}}{4 \nu \beta}\left|\beta^{2}-\alpha^{2}\right|^{2}|Au|^{2}\|u\|^{2}.
\end{align}

Denoting
$$ \xi(t)=|g(t)|^{2}+\beta^{2}\|g(t)\|^{2}, \; \gamma=\frac{\nu \lambda_{1}}{2},$$
$$\delta (t)=\frac{5\left|\beta^{2}-\alpha^{2}\right|^{2}}{\beta}\left(\frac{1}{\nu \beta}|u_{t}|^{2}+\frac{\nu}{\beta}|Au|^{2}+\frac{c^{2}}{2 \nu}|Au|^{2}\|u\|^{2}\right),$$
and applying \eqref{Grn} for all $t \geqslant t_0$, we have

$$
\frac{d}{d t} \xi(t)+\gamma \xi(t) \leqslant \delta(t).
$$

Moreover, from Lemmas \ref{LEM1} and \ref{LEM2}, we obtain

$$
\begin{aligned}
&\int_{t}^{t+T} \delta(s)ds\\
=&\frac{5\left|\beta^{2}-\alpha^{2}\right|^{2}}{\beta}\Bigg(\frac{1}{\nu \beta} \int_{t}^{t+T}\left|\frac{du}{dt}(s)\right|^{2} ds+\frac{\nu}{\beta} \int_{t}^{t+T}|Au(s)|^{2} ds\\
& +\frac{c^{2}}{2 \nu} \int_{t}^{t+T}|Au(s)|^{2}\left\| u(s)\right\|^{2} ds\Bigg)\\
\leqslant &\frac{5\left|\beta^{2}-\alpha^{2}\right|^{2}}{\beta} \left[\frac{1}{\nu \beta} \left(\frac{6 M_{1}}{\nu \alpha^{2}}+\frac{6T}{\nu^{2} \lambda_{1} \alpha^{2}} \sup_{s \geqslant 0} |f(s)|^{2} \right) \left(\nu^2 +\frac{c^2 M_{1}}{16 \pi^{4} \lambda_{1}^{1 / 2} \alpha^{2}} \right) \right.\\
& +\frac{1}{\nu \beta}\left(\frac{3 c^{2}}{16 \pi^{4} \alpha^{4} \lambda_{1}^{1 / 2}} \cdot \frac{M_{1}}{\alpha^{2}} \cdot M_{1}\right) T+\frac{3}{\nu \beta} T \sup_{s \geqslant 0}|f(s)|^2 \\
& \left.+\frac{\nu}{\beta}\left(\frac{M_{1}}{\nu \alpha^{2}}+\frac{T}{\nu^{2} \alpha^{2} \lambda_{1}} \sup_{s \geqslant 0}|f(s)|^{2}\right)+\frac{c^{2}}{2 \nu} \frac{M_{1}}{\alpha^{2}}\left(\frac{M_{1}}{\nu \alpha^{2}}+\frac{T}{\lambda_{1}{\nu }^{2} \alpha^{2}} \sup_{s \geqslant 0}  |f(s)|^{2}\right)\right].
\end{aligned}
$$
Choosing $T=\left(\nu \lambda_{1}\right)^{-1}>0$, we have

\begin{align}
&\int_{t}^{t+\frac{1}{\lambda_{1} \nu }} \delta(s) ds \\
\leqslant &\ds\frac{5\left|\beta^{2}-\alpha^{2}\right|^{2}}{\beta}\left[\frac{1}{\nu \beta}\left(\frac{6 M_{1}}{\nu \alpha^{2}}+\frac{6}{\nu^{3} \lambda_{1}^{2} \alpha^{2}} \sup_{s \geqslant 0} |f(s)|^{2}\right)\left(\nu^{2}+\frac{c^2 M_{1}}{16 \pi^{4} \lambda_{1}^{1 / 2} \alpha^{2}}\right)\right. \nonumber\\
& +\ds\frac{3 c^{2} M_{1}^{2}}{16 \nu^2 \beta \pi^{4} \alpha^{6} \lambda_{1}^{3 / 2}}+\frac{3}{\nu^{2} \lambda_{1} \beta} \sup_{s \geqslant 0} | f(s) |^{2}+\frac{\nu}{\beta}\left(\frac{M_{1}}{\nu \alpha^{2}}+\frac{1}{\nu^{3} \alpha^{2} \lambda_{1}^{2}} \sup_{s \geqslant 0} |f(s)|^{2}\right) \nonumber\\
& \left.+\ds\frac{c^{2} M_{1}}{2 \nu \alpha^{2}}\left(\frac{M_{1}}{\nu \alpha^{2}}+\frac{1}{\lambda_{1}^{2} \nu^{3} \alpha^{2}} \sup_{s \geqslant 0} |f(s)|^{2}\right)\right]:=M_\alpha.\label{M2}
\end{align}

Therefore, by applying Lemma \ref{GronwallVersion}, we have the following estimate for the error between the original $\alpha$-model and the approximation $\beta$-model solutions:
$$
|g(t)|^{2}+\beta^{2}\|g(t)\|^{2} \leqslant e^{-\frac{\lambda \nu t}{2}}\left(|g(0)|^{2}+\beta^{2}\|g(0)\|^{2}\right)+M_{\alpha}\left(\frac{e}{e^{1 / 2}-1}\right),
$$
with $M_\alpha$ defined in \eqref{M2}.
\end{proof}

\section{Numerical Simulations}\label{NumSim}
In this section, we conduct numerical experiments to illustrate and verify the theoretical results on the convergence as stated in Theorem \ref{teoMLalpha}. In order to complete this task, we numerically solve the ML-$\alpha$ model in a three-dimensional domain. We use a newly developed Python package ``Dedalus", which supports symbolic entry for equations and conditions. Moreover, Dedalus utilizes spectral methods for solving partial differential equations and is particularly convenient for problems with periodic boundary conditions. Using Dedalus, we perform the following two sets of numerical simulations: one is with initial conditions without a random component to assess the impact of $\eta$ and the other one is with initial conditions with a random component to assess the impact of $\eta$. For each scenario, we have provided two types of graphical results. One presents the normalized difference between the solutions from the original system and the data assimilation system. In these graphs, a decreasing trend indicates convergence, while an increasing trend represents divergence. The other graph displays the velocity contours for both the original system and the data assimilation system at the initial and final time steps. In cases of convergence, even if the two systems start differently, their velocity contours become similar by the end. Conversely, in divergent cases, the velocity contours remain distinct. The following provide the details.

\subsection{Testing the impact of $\eta$-without random initial conditions}

Choosing a domain $\Omega=[0, 1]^3$, we have the initial conditions for the original system (\ref{MLalpha}) to be $u=(u_0,\ v_0,\ w_0):$
$$u_0=0.05\sin{2\pi xz},$$
$$v_0=0.05\sin{2\pi xy},$$
$$w_0=0.05\sin{2\pi yz}.$$

The initial conditions for the assimilated model (\ref{MLalphaCDA}) is taken to be $w=(\hat{u}_0,\ \hat{v}_0,\ \hat{w}_0):$
$$\hat{u}_0=0.05\sin{\pi x}\cos{\pi y},$$
$$\hat{v}_0=0.05\sin{\pi y}\cos{\pi z},$$
$$\hat{w}_0=0.05\sin{\pi z}\cos{\pi x}.$$
In practice, this initial condition is arbitrary and can be set to anything within the domain.

We carefully select the parameters so they satisfy the conditions for regularity in Theorem \ref{aaa1}
\begin{eqnarray}
 \eta\leq \frac{\nu}{2} \min \bigg\{\ds\frac{1}{c_1^2h^2},\frac{\beta^2}{c_2^2h^4}\bigg\}   \mbox{ and } \eta^2c_2^2h^4<\ds\frac{\nu^2}{4}, \nonumber
\end{eqnarray}

as well as check for the hypotheses given in Theorem \ref{teoMLalpha}, i.e. 
\begin{enumerate}
    \item $ \ds\frac{3 \eta}{4}-\frac{15^{3} M_{1}^{2}c^4}{2^{11} \nu^{3} \alpha^{4}}>0 $ $\Longleftrightarrow$ $ \eta>C_1 := \frac{4}{3} \cdot \frac{15^{3} M_{1}^{2}c^4}{2^{11} \nu^{3} \alpha^{4}}, $ 
    \item $ \eta c_{1}^2 h^{2}+\ds\frac{5 \eta^{2} \beta^{2}}{2 \nu} c_{1}^2 h^{2}+\frac{30^{3} c^{4} M_{1}^{2}}{4^{4} \nu^{3} \alpha^{4}}-\eta \beta^{2}<\frac{\nu}{4}, $
    \item $ \eta c_{2}^2 h^{4}+\ds\frac{5 c_{2}^2 \eta^{2} \beta^{2} h^{4}}{2 \nu}<\frac{\nu \beta^{2}}{4}. $
\end{enumerate}
Here, the constants are $c={\frac{4}{3\sqrt{3}}}^{3/4}$ \cite{Galdi}, $c_1=\sqrt{32}$, and $c_2=2$ \cite{albanez2016continuous}. We first fix $\nu=0.75$ and $\alpha=0.3$. $M_1$ depends on the initial conditions and force. Here, we take the force to be 0 and $M_1= 0.00339$, so we have $C_1=0.00739$. We then compare two scenarios: one with $\eta=1.5>C_1$
  and the other with $\eta=0.0001<C_1$. Once $\eta$ is chosen, we choose $h=0.043$ and $\beta=0.35$ so condition (\ref{kkkk}) and the last two hypotheses 2 and 3 are satisfied. 
  The graphical results on the difference are presented in Figures 1 (high $\eta$) and 4 (low $\eta$). In all these error plots, the x-axis represents time, while the y-axis shows the logarithm of the normalized difference between the solution from the original system and the data assimilation system. The velocity contours are presented in Figures 2 and 3 (high $\eta$)
  and 5 and 6 (low $\eta$). In these contour graphs, we compare the original and the assimilated systems at the beginning and at the end time steps.


\begin{figure}
\centering
\includegraphics[totalheight=7cm]{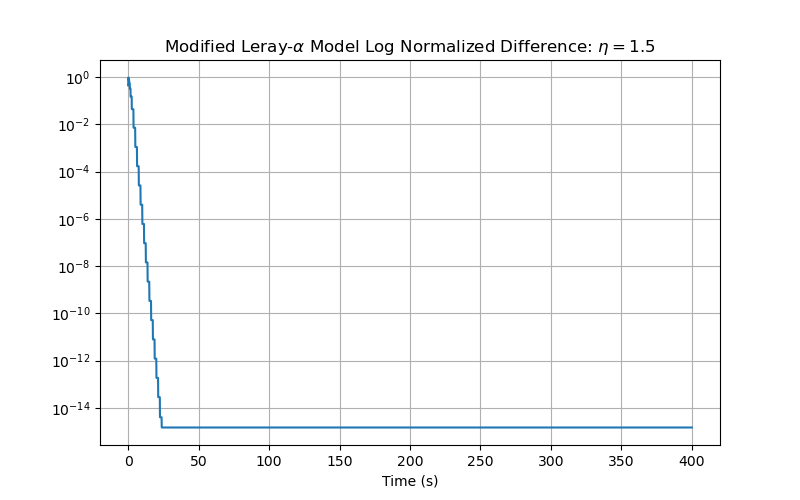}
    \caption{Error plot of modified Leray-$\alpha$ model with high $\eta$ value-without random initial conditions case.}
    \label{fig1}
\end{figure}
\begin{figure}
\centering
    \includegraphics[totalheight=6cm]{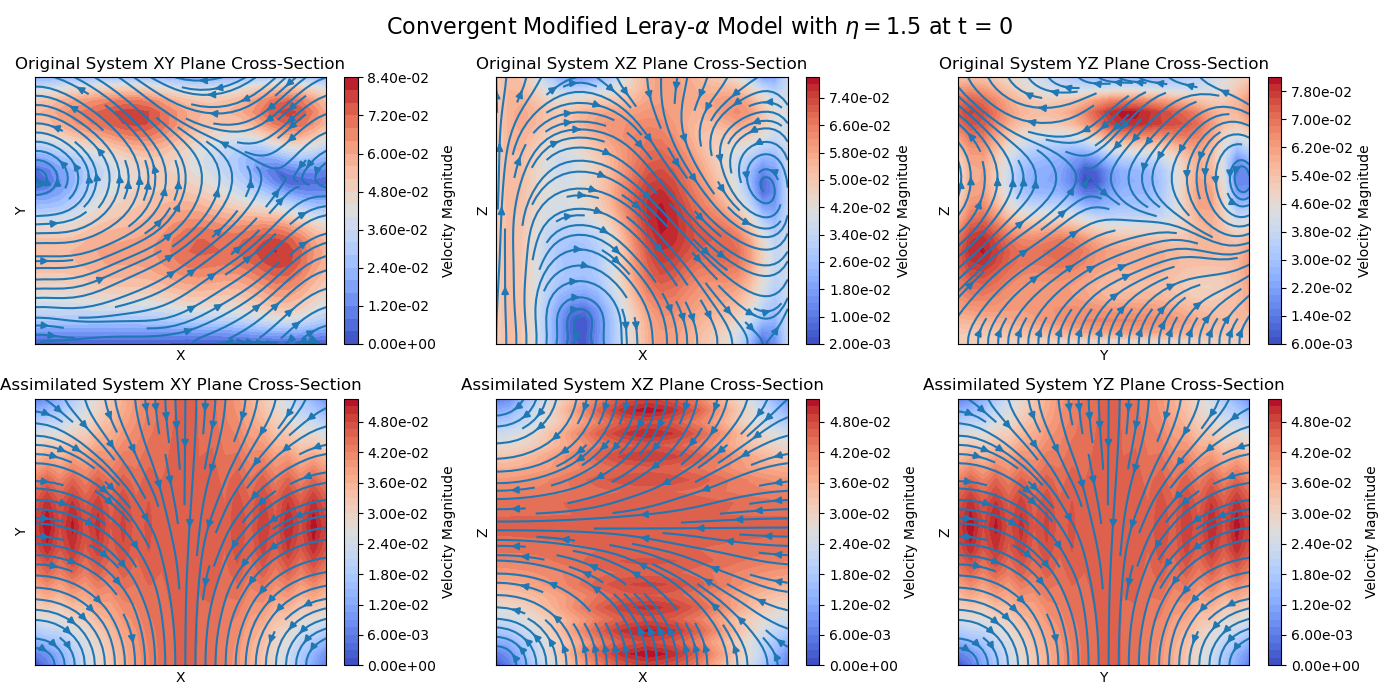}
    \caption{Velocity contour of modified Leray-$\alpha$ model with high $\eta$ value-without random initial conditions case at $t=0$.}
    \label{fig2}
\end{figure}
\begin{figure}
\centering
    \includegraphics[totalheight=6cm]{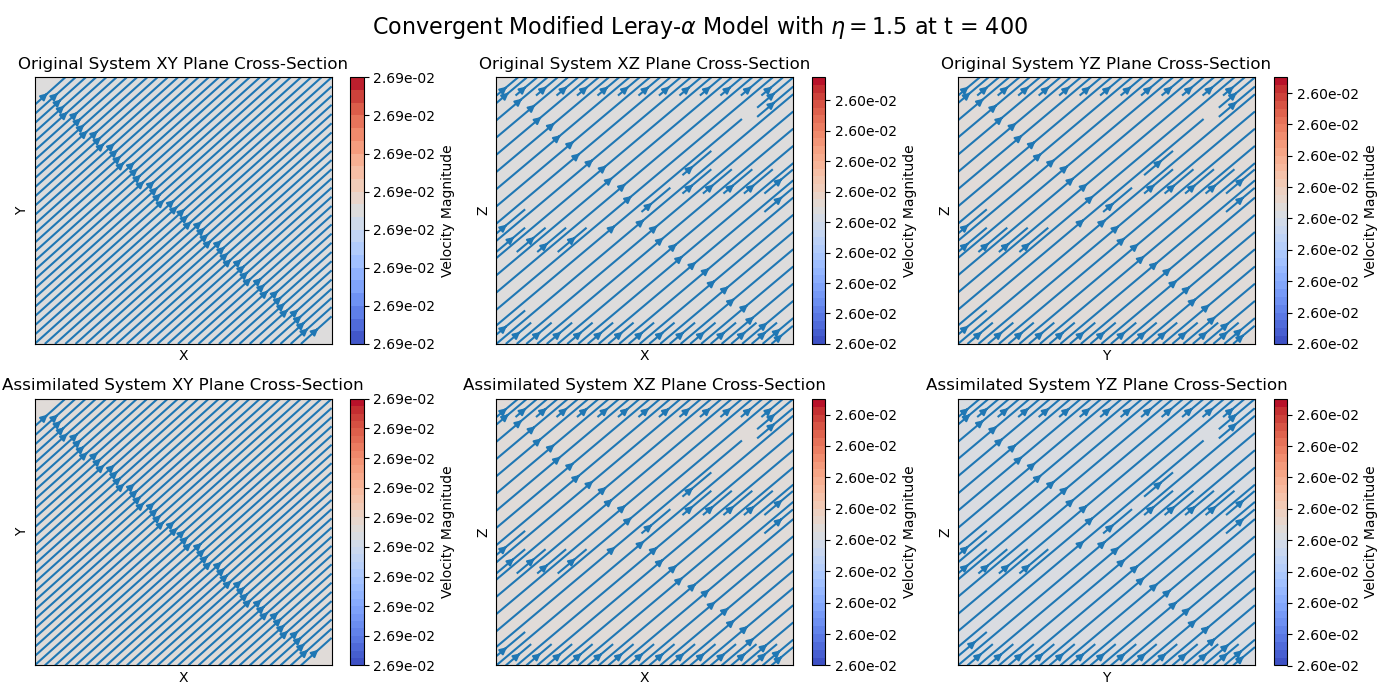}
    \caption{Velocity contour of modified Leray-$\alpha$ model with high $\eta$ value-without random initial conditions case at $t=400$.}
    \label{fig3}
\end{figure}
\begin{figure}
\centering
    \includegraphics[totalheight=7cm]{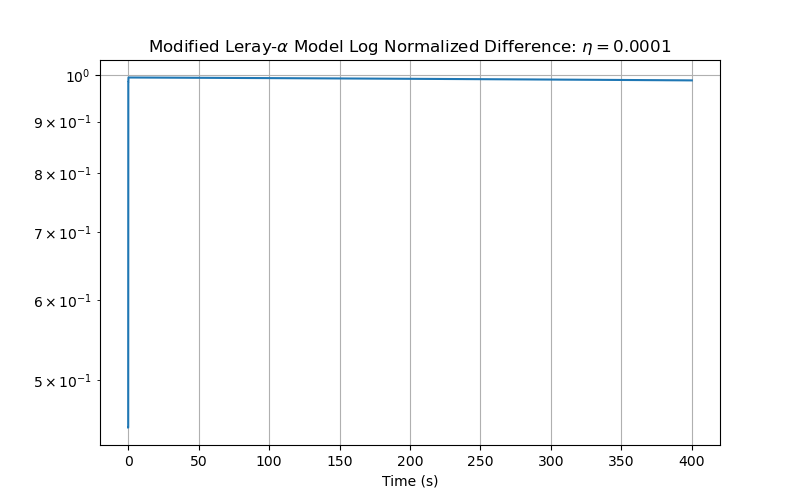}
    \caption{Error plot of modified Leray-$\alpha$ model with low $\eta$ value-without random initial conditions case.}
    \label{fig4}
\end{figure}
\begin{figure}
\centering
    \includegraphics[totalheight=6cm]{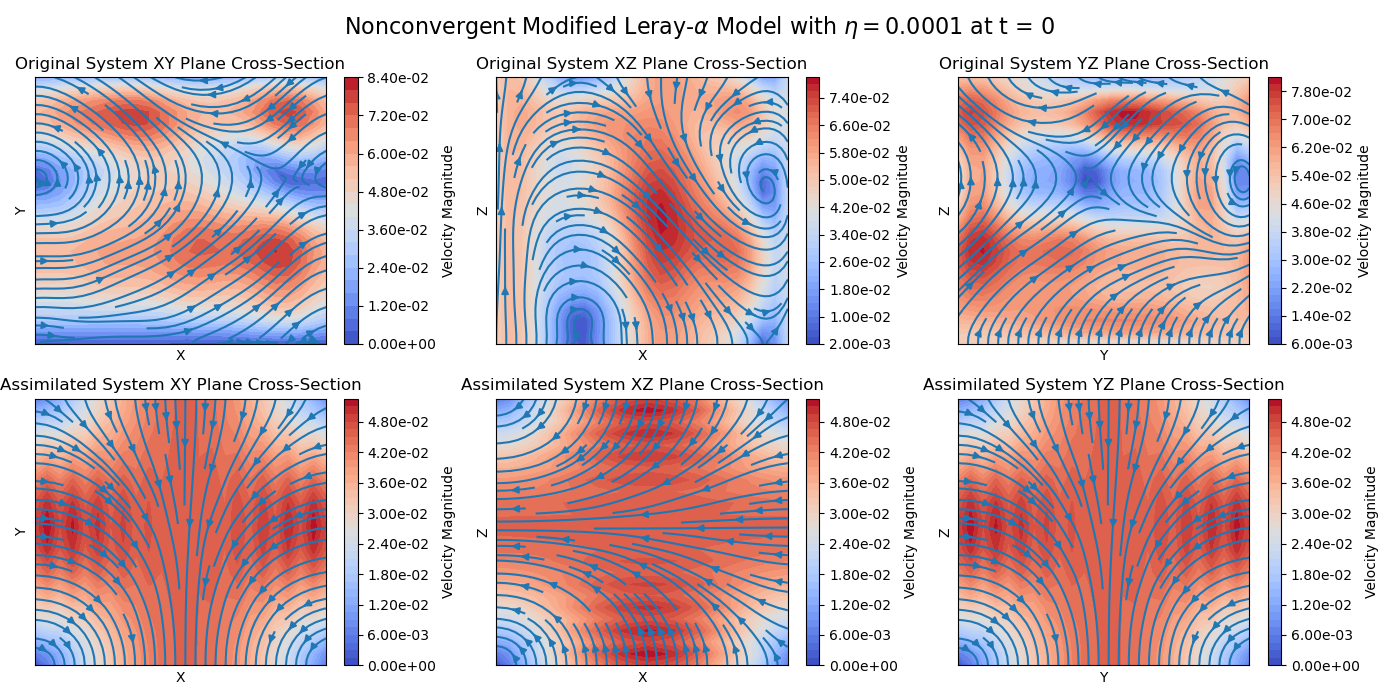}
    \caption{Velocity contour of modified Leray-$\alpha$ model with low $\eta$ value-without random initial conditions case at $t=0$.}
    \label{fig5}
\end{figure}
\begin{figure}
\centering
    \includegraphics[totalheight=6cm]{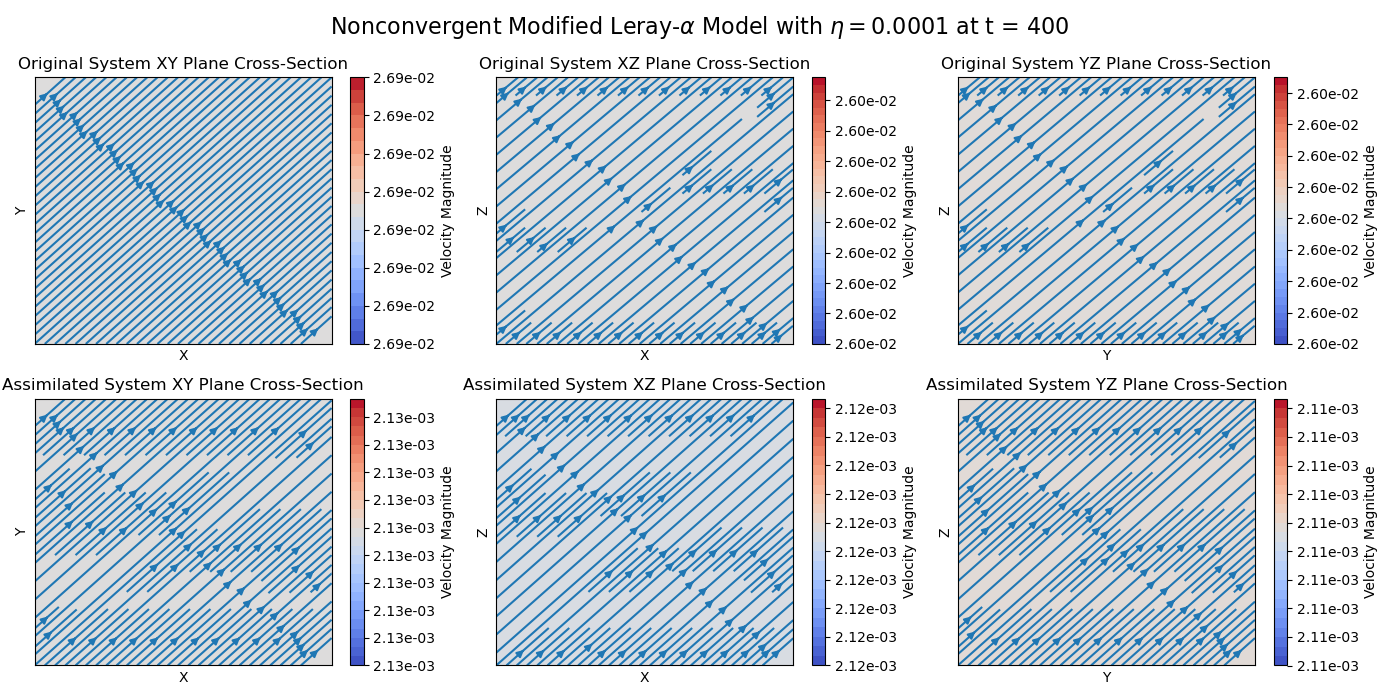}
    \caption{Velocity contour of modified Leray-$\alpha$ model with low $\eta$ value-without random initial conditions case at $t=400$.}
    \label{fig6}
\end{figure}
\subsection{Testing the impact of $\eta$-with random initial conditions}
The domain is $\Omega=[0, 1]^3$ and the initial conditions for the original system (\ref{MLalpha}) has a random component
 which is $u=(u_0,\ v_0,\ w_0):$
$$u_0=0.05\sin{2\pi xz}+0.02*X -0.01,$$
$$v_0=0.05\sin{2\pi xy}+0.02*X -0.01,$$
$$w_0=0.05\sin{2\pi yz}+0.02*X -0.01,$$
where $X$ is a random variable drawn from a uniform distribution.

The initial conditions for the assimilated model (\ref{MLalphaCDA}) is taken to be $w=(\hat{u}_0,\ \hat{v}_0,\ \hat{w}_0):$
$$\hat{u}_0=0.05\sin{\pi x}\cos{\pi y},$$
$$\hat{v}_0=0.05\sin{\pi y}\cos{\pi z},$$
$$\hat{w}_0=0.05\sin{\pi z}\cos{\pi x}.$$
Here, we have $\nu=0.75$ and $\alpha=0.3$. $M_1=0.00355$, $h=0.043$, and $\beta=0.35$. We compare the results when $\eta=1.5>C_1\approx 0.00811$ (results in 7-9) and $\eta=0.0001<C_1\approx 0.02626$ (results in 10-12). 
Note that, due to the random component in our initial conditions, each run yields a different $C_1$ value; however, these values are very close to one another.
\begin{figure}
\centering
    \includegraphics[totalheight=7cm]{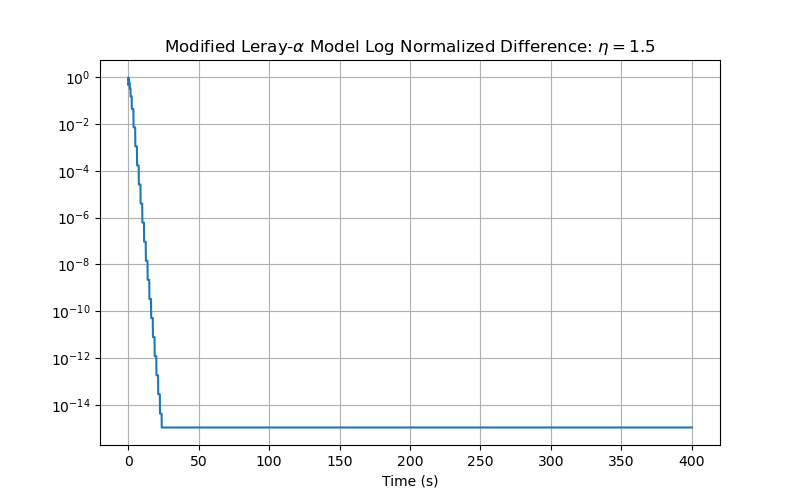}
    \caption{Error plot of modified Leray-$\alpha$ model with high $\eta$ value-with random initial conditions case.}
    \label{fig13}
\end{figure}
\begin{figure}
\centering
    \includegraphics[totalheight=6cm]{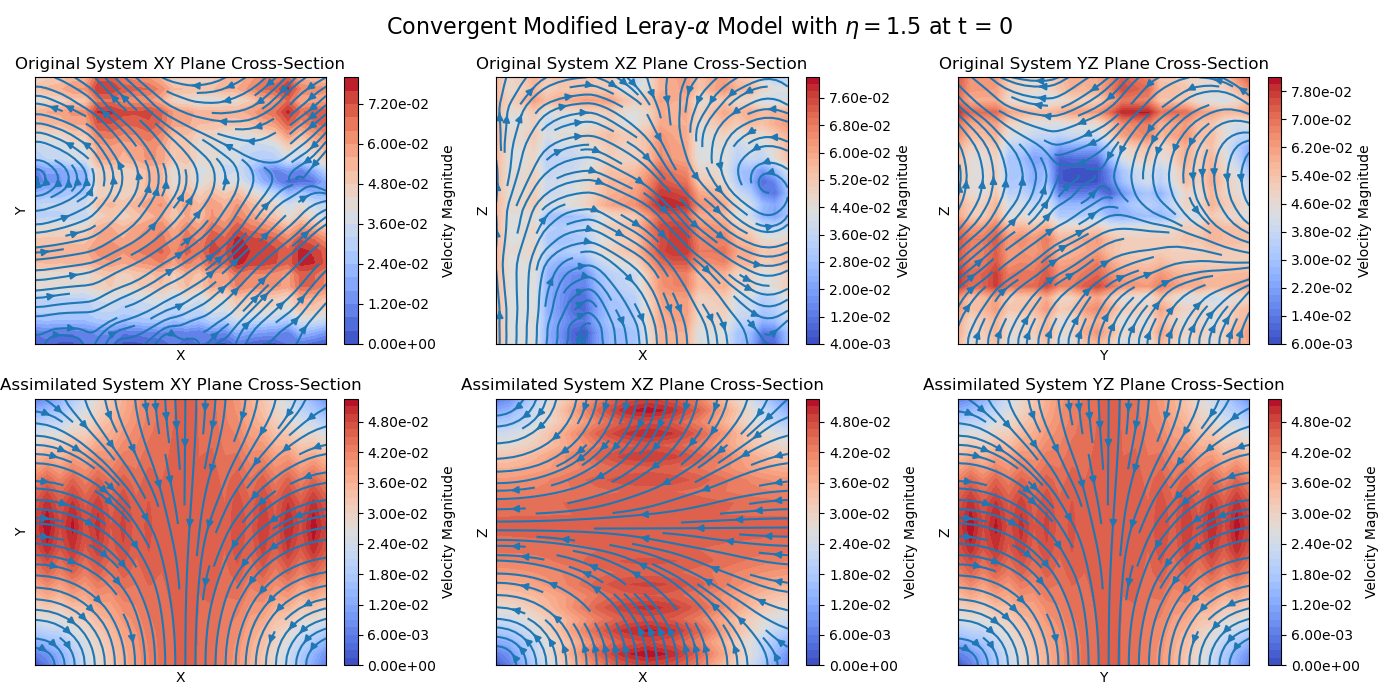}
    \caption{Velocity contour of modified Leray-$\alpha$ model with high $\eta$ value-with random initial conditions case at $t=0$.}
    \label{fig14}
\end{figure}
\begin{figure}
\centering
    \includegraphics[totalheight=6cm]{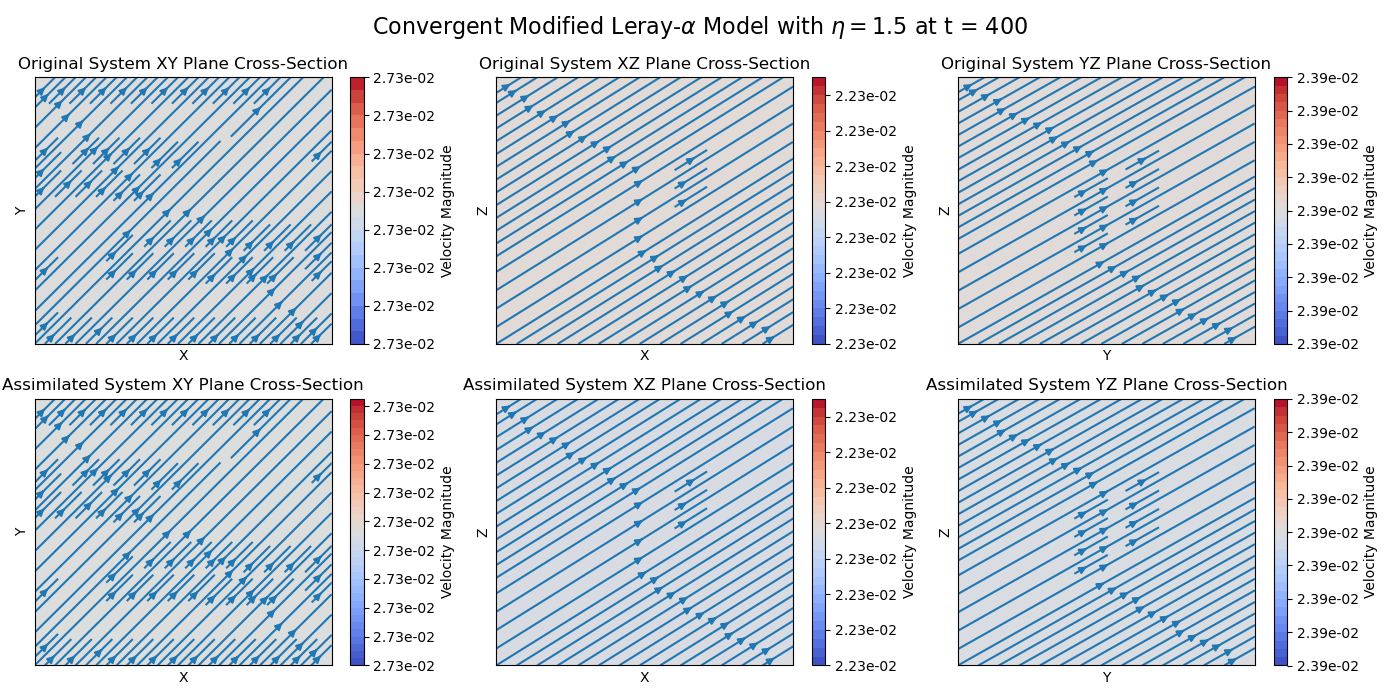}
    \caption{Velocity contour of modified Leray-$\alpha$ model with high $\eta$ value-with random initial conditions case at $t=400$.}
    \label{fig15}
\end{figure}
\begin{figure}
\centering
    \includegraphics[totalheight=7cm]{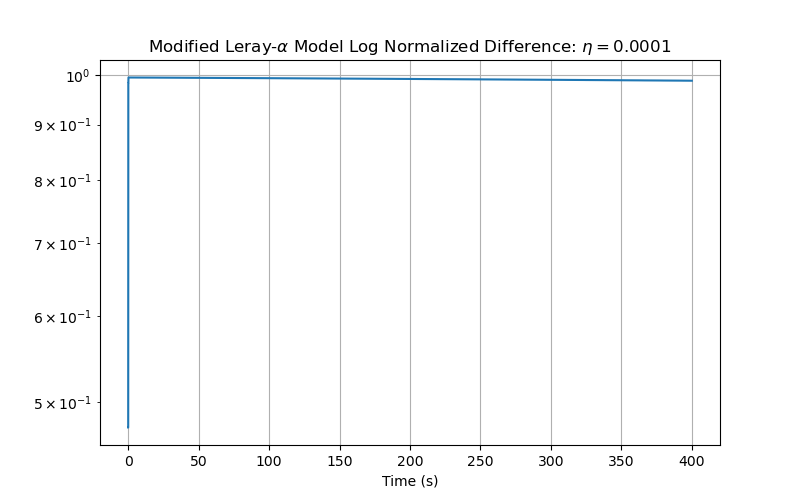}
    \caption{Error plot of modified Leray-$\alpha$ model with low $\eta$ value-with random initial conditions case.}
    \label{fig16}
\end{figure}
\begin{figure}
\centering
    \includegraphics[totalheight=6cm]{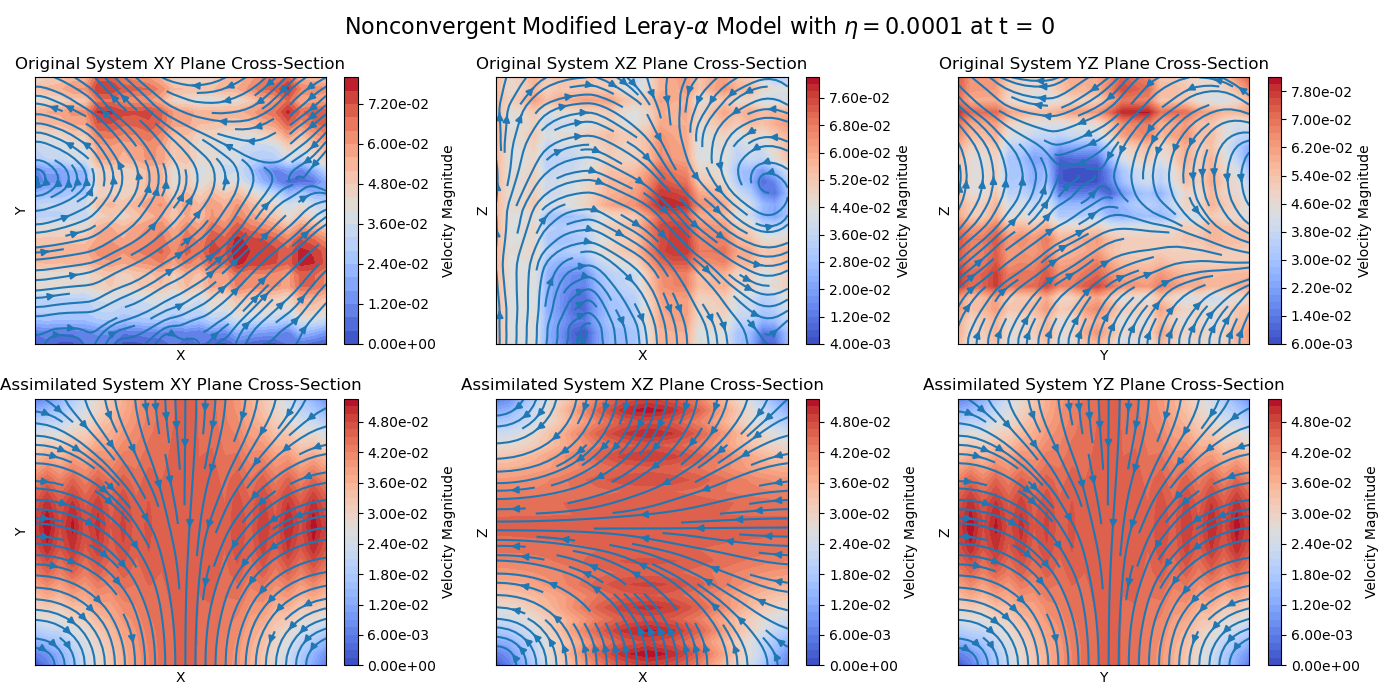}
    \caption{Velocity contour of modified Leray-$\alpha$ model with low $\eta$ value-with random initial conditions case at $t=0$.}
    \label{fig17}
\end{figure}
\begin{figure}
\centering
    \includegraphics[totalheight=6cm]{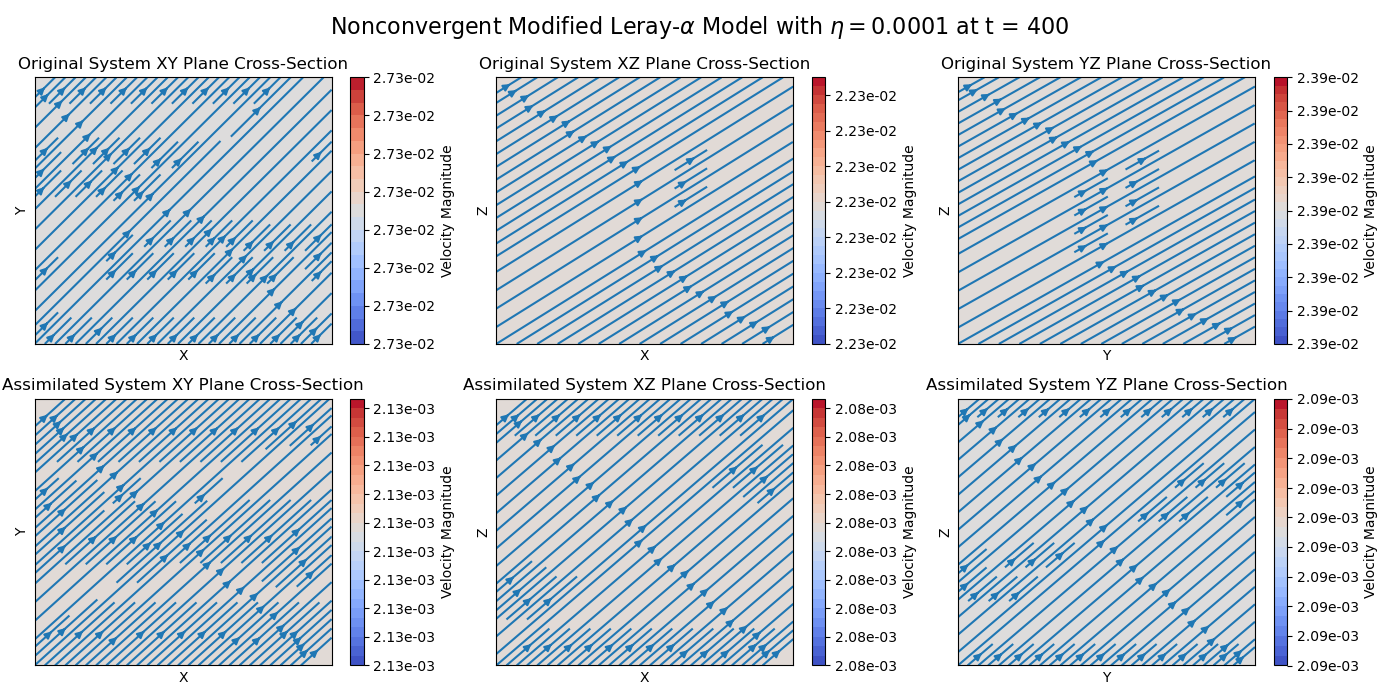}
    \caption{Velocity contour of modified Leray-$\alpha$ model with high $\eta$ value-with random initial conditions case at $t=400$.}
    \label{fig18}
\end{figure}
\subsection{Discussions on the numerical computations} 
When exploring the behavior of the ML-$\alpha$ data assimilation model, it is found that
convergence is always achieved when the three hypotheses are met. If this is the case,
the speed of convergence seems to decrease with the value of $\eta$. Besides testing $\eta=1.5$ as showing above, we have also tested $\eta=1$ which reaches the
computational floating point error of order of magnitude $10^{-15}$ after about 40 seconds. At this point, the models will not converge to each other any further due to the numerical truncation
error that occurs in the 16 bit floating point value. When $\eta$ is set to 1.5, the convergence of the
original and assimilated system reach an order of magnitude of $10^{-15}$ after about 20 seconds. When the conditions are not met and $\eta$ is very small, the systems very quickly diverge to different
average flow magnitudes. Moreover, in this numerical test, the initial conditions of the original system are not set with any notable
similarity to those of the assimilated system. We also test these same conditions with an additional randomization term added to the original
system. This serves to help us evaluate the role of random perturbations in the systems, but it is evident
that this term has very little or no impact in the systems’ convergence. We do observe some
aliasing behavior when $\eta$ is set to 1 (and is generally close to the $\eta$ condition). This behavior is smoothed out when $\eta$ is raised slightly back up to 1.5. In both randomized and smooth cases, convergence is achieved when all three convergence
conditions are met.
\section{Conclusions}\label{conclusion}
In this work, we proposed a continuous data assimilation algorithm for the ML-$\alpha$ model. We began by proving the global well-posedness of the assimilated system. Unlike previous studies in the literature, our approximate solution is with the length scale parameter $\alpha>0$ being considered unknown. We demonstrated that, under suitable conditions, the approximate solution converges to the true solution, with an error influenced by factors such as viscosity, the forcing term, and estimates of the $L^2$-norm of the true solution and its derivatives. Additionally, the error is evidently affected by the difference between the true and approximating parameters. Numerical simulations were provided to validate our theoretical results.

In the future, we aim to implement a parameter recovery algorithm for the ML-$\alpha$ model. Specifically, we plan to design an algorithm capable of recovering the unknown parameter $\alpha$. Moreover, we plan to conduct a comparison between various turbulence models with continuous data assimilation algorithms.\\

\textbf{Acknowledgment}
\vspace{3pt}

Samuel Little and Jing Tian's work is partially supported by the NSF LEAPS-MPS Grant $\#2316894$.

\end{document}